\documentclass[a4paper, 11pt]{article}
\usepackage[utf8]{inputenc}
\usepackage[T1]{fontenc}
\usepackage{textcomp}
\usepackage{amsmath, amssymb}
\usepackage{amsthm}
\usepackage[hmargin= 2.5 cm]{geometry}
\usepackage[colorlinks=true]{hyperref}
\usepackage[french]{babel}
\usepackage{mathrsfs}
\title{Optimal bounds for the growth of Sobolev norms of solutions of a quadratic Szeg\H{o} equation}
\author{Joseph \bsc{Thirouin}}
\date{October 3\up{rd}, 2017}

\DeclareMathOperator{\im}{Im}
\DeclareMathOperator{\re}{Re}
\DeclareMathOperator{\tr}{Tr}

\DeclareMathOperator{\rg}{rg}

\newtheorem{thm}{Theorem}
\newtheorem{prop}{Proposition}[section]
\newtheorem{lemme}[prop]{Lemma}
\newtheorem{cor}[prop]{Corollary}
\theoremstyle{definition}
\newtheorem*{déf}{Definition}
\theoremstyle{remark}
\newtheorem{rem}{Remark}

\usepackage{enumitem}
\setenumerate[1]{label=(\roman*)}

\begin{document}

\renewcommand{\refname}{Bibliography}
\renewcommand{\abstractname}{Abstract}
\maketitle

\begin{abstract}
In this paper, we study a quadratic equation on the one-dimensional torus :
\[i \partial_t u = 2J\Pi(|u|^2)+\bar{J}u^2, \quad u(0, \cdot)=u_0,\]
where $J=\int_\mathbb{T}|u|^2u \in\mathbb{C}$ has constant modulus, and $\Pi$ is the Szeg\H{o} projector onto functions with nonnegative frequencies. Thanks to a Lax pair structure, we construct a flow on $BMO(\mathbb{T})\cap \im\Pi$ which propagates $H^s$ regularity for any $s>0$, whereas the energy level corresponds to $s=1/2$. Then, for each $s>1/2$, we exhibit solutions whose $H^s$ norm goes to $+\infty$ exponentially fast, and we show that this growth is optimal.\par
\textbf{MSC 2010 :} 37K10, 37K40, 35B45.\par
\textbf{Keywords :} Hamiltonian systems, Szeg\H{o} equation, a priori estimates, Lax pair.
\end{abstract}

\section{Introduction}
Quantifying the growth of Sobolev norms of solutions to a Hamiltonian PDE is a good way of understanding how fast the energy moves from lower to higher frequencies and conversely, a phenomenon which is typical of what is known as \emph{wave turbulence}. Upper bounds on such a growth are usually known \cite{Bourgain, planchon, Staffilani, thirouin1}, but it is very difficult to find out whether they are optimal or not, unless explicit growing solution are known.

A setting combining Hamiltonian properties, explicit computations, and growth of Sobolev norms has been succesfully introduced in 2010 by Gérard and Grellier \cite{ann}. The authors designed a simple toy model, called the \emph{cubic Szeg\H{o} equation}, which turns out to be completely integrable, in the sense that action-angle variables can be found, making computations possible (though not obvious). On the other hand, this system discloses instability, and it is possible to prove the existence of generic turbulent orbits, with Sobolev norms growing more than polynomially in time and oscillating back and forth \cite{livrePG}. This toy model eventually enlightens the large-time behaviour of non-integrable hamiltonian systems, for which it enables to find turbulent solutions through a study of resonances \cite{Xu2} (see also \cite{hani, HPTV} for the same ideas in the case of the nonlinear Schrödinger equation on $\mathbb{T}^2$). Note that the cubic Szeg\H{o} equation also looks similar to physically relevant equations, such as the one studied in \cite{Bizon}. However, in all these situations, the optimality of the bounds on the growth of solutions remains to be established.

As observed in \cite{thirouin1}, a priori estimates are much simpler to prove when nonlinearities are only quadratic. This gives the idea of exploiting the framework of the cubic Szeg\H{o} equation in an apparently simpler case, where we only take a \emph{cubic} Hamiltonian (instead of quartic). More specifically, we choose our phase space to be the closed subset of $L^2(\mathbb{T})$, called $L^2_+(\mathbb{T})$, consisting of all functions with only nonnegative frequencies :
\[ L^2_+(\mathbb{T}) = \left\lbrace u\in L^2(\mathbb{T}) \: \middle| \: \hat{u}(k)=0, \:\forall k<0\right\rbrace . \]
$L^2$ is endowed with its standard Hilbert structure :
\[ (f\vert g):=\frac{1}{2\pi}\int_0^{2\pi}f(e^{ix})\overline{g(e^{ix})}dx.\]
We denote by $\Pi : L^2(\mathbb{T})\to L^2_+(\mathbb{T})$ the orthogonal projection onto $L^2_+$. $\Pi$ is usually called the \emph{Szeg\H{o} projector}. For any subspace $G$ of $L^2(\mathbb{T})$, we will use the notation $G_+:=G\cap L^2_+$.

\begin{rem}
The space $L^2_+(\mathbb{T})$ is also called the Hardy space on the disc, and can be identified with the space of holomorphic functions on the open unit disc $\mathbb{D}\subset\mathbb{C}$ whose trace on the boundary $\partial \mathbb{D}$ is in $L^2$.
\[ u(x)=\sum_{n=0}^\infty \hat{u}(n)e^{inx}\ \stackrel{\sim }{\longmapsto} \ u(z)=\sum_{n=0}^\infty \hat{u}(n)z^n, \quad \lim_{r\to 1^-} \frac{1}{2\pi}\int_0^{2\pi}|u(re^{i\theta})|^2d\theta <\infty.\]
In the sequel, we shall use either notation to designate the points of our phase space.
\end{rem}

On a dense subset of $L^2_+$, we can define a functional $E$, that will be taken as our Hamiltonian (our energy) :
\[ E(u):= \frac{1}{2}\left| \int_\mathbb{T} |u|^2u \right| ^2 = \frac{1}{2}|(u^2\vert u)|^2,\]
With respect to the natural symplectic structure given by the 2-form $\omega (u,v) := \im (u\vert v)$, the Hamiltonian equation deriving from $E$ reads $\partial_t u=-i\nabla E(u)$, \textit{i.e.}
\begin{equation*}
i\partial_tu=2(u^2|u)\Pi(|u|^2)+(u|u^2)u^2,
\end{equation*}
Let us here simplify the notation, calling $J$ the factor $(u^2|u)$ (as a reminiscence of the $J_3$ of Szeg\H{o} hierarchy introduced in \cite{ann}), so that $E=\frac{1}{2}|J|^2$ and the evolution of $u$ is given by
\begin{equation}\label{quad}
i\partial_t u=2J\Pi(|u|^2)+\bar{J}u^2.
\end{equation}

Because of the invariances of the Hamiltonian $E$, we know that there are at least three conservation laws for smooth solutions of \eqref{quad} :
\begin{align*}
Q(u) &:= (u|u), & \text{(the mass)}\\
M(u) &:= (Du|u), \: D:=-i\partial_x, & \text{(the momentum)}\\
E(u) &= \tfrac{1}{2}|(u^2\vert u)|^2=\tfrac{1}{2}|J|^2. & \text{(the energy)}
\end{align*}
In particular, the modulus of $J$ is conserved by the flow, which makes system \eqref{quad} look like a ``quadratic Szeg\H{o} equation''. In fact, we will show the existence of infinitely many conservation laws for equation \eqref{quad}, proving that it is associated with a Lax pair structure (see theorem \ref{lax_thm} below).

Notice also that the $H^{1/2}$ norm of a smooth solution remains bounded, since for $u\in L^2_+$,
\[ (Q+M)(u)=\sum_{n\geq 0}(1+n)|\hat{u}(n)|^2\simeq \|u\|_{H^{1/2}}^2.\]
This is what makes $H^{1/2}_+(\mathbb{T})$ a natural space to define a flow. We will call it the \emph{energy space}.

However, it turns out that we can even define solutions below the $H^{1/2}$ regularity, which is consistent with the fact that equation \eqref{quad} does not involve any derivative in the $x$-variable. Following the approach of \cite{GKoch}, and using the Lax pair structure, we will prove that \eqref{quad} admits a flow on $BMO_+(\mathbb{T}):=BMO\cap L^2_+(\mathbb{T})$, where $BMO$ denotes the usual space of John and Nirenberg of functions such that
\[ \sup_{I\subset \mathbb{T}}\: \frac{1}{|I|}\int_I \left| f-\frac{1}{|I|}\int_I f\right| <+\infty,\]
where the supremum is taken on intervals $I$ of $\mathbb{T}$.

\vspace*{1em}
The main theorem of this paper is the following :
\begin{thm}\label{main}
Let $s\geq 0$ be any nonnegative real number. Let $u_0\in BMO_+\cap H^s_+(\mathbb{T})$. 
\begin{enumerate}
\item Then there exists a unique $u\in C(\mathbb{R},H^s_+)\cap C_{w*}(\mathbb{R},BMO_+)$ solution to
\[ i\partial_tu=2J\Pi(|u|^2)+\overline{J}u^2,\quad u(0)=u_0.\]
This solution stays bounded in the $BMO$ norm, and (when $s\geq \frac{1}{2}$) in the $H^{1/2}$ norm.

In addition, for each $t\in \mathbb{R}$, the mapping $u_0\mapsto u(t)$ is Lipschitz continuous on the ball $\mathcal{B}_{BMO}(R):=\{v\in BMO_+(\mathbb{T})\mid \|v\|_{BMO}\leq R\}$ endowed with the $L^2$ distance.

\item Let $s'\in (0,s]$ (if $s<\frac{1}{2}$), or $s'\in (\frac{1}{2},s]$ (if $s\geq \frac{1}{2}$). Then there are constants $C, B>0$ such that
\[ \forall t\in\mathbb{R},\quad \|u(t)\|_{H^{s'}}\leq Ce^{B|t|}.\]
$B$ can be taken to depend only on $s'$ and on $\|u_0\|_{BMO}$ (when $s'<1$), on $\|u_0\|_{H^{1/2}}$ (when $s'=1$) or on $\|u_0\|_{H^{s'}}$ (when $s'>1$).

Moreover, in the case when $s\geq \frac{1}{2}$, these estimates are \emph{optimal}, \emph{i.e.} for each $s\geq\frac{1}{2}$, there exists $u_0\in BMO_+\cap H^s$ such that $\|u(t)\|_{H^{s'}}$ grows exponentially fast for each $s'\in (\frac{1}{2},s]$.
\end{enumerate}
\end{thm}

Note that when $s\geq \frac{1}{2}$, $H^s_+(\mathbb{T})\subset BMO_+(\mathbb{T})$. In particular, Theorem \ref{main} yields that equation \eqref{quad} is well-posed in the energy space.

Furthermore, when $s=0$, $BMO_+\cap H^0(\mathbb{T})=BMO_+(\mathbb{T})$, hence the above theorem states the existence of a flow on $BMO_+$ which propagates additional regularity. This phenomenon strongly hinges on the quadratic nature of equation \eqref{quad}. As shown in \cite{GKoch}, it is also possible to define a flow on $BMO_+$ for the cubic Szeg\H{o} equation, but in that case it is not known whether it preserves $H^s$ regularity or not, when $0<s<1/2$.

\vspace*{1em}
Let us turn to the optimality of the a priori estimates. As we mentioned before, the existence of turbulent trajectories can be shown thanks to explicit computations relying on the integrability of the equation. Indeed, a consequence of the Lax pair structure will be that the set
\[ \mathcal{L}(1):=\left\lbrace u(x)=b+\frac{ce^{ix}}{1-pe^{ix}}\, \middle| \, b,c,p\in\mathbb{C},|p|<1\right\rbrace \]
is stable by the flow of \eqref{quad}. On this set, \eqref{quad} will reduce to a system of couple ODEs, from which we will be able to derive a necessary and sufficient condition for norm explosion. We will prove the following proposition :

\begin{prop}\label{blowup}
Suppose that $u$ is a non-constant solution of \eqref{quad} such that $u(0)=u_0\in \mathcal{L}(1)$. Then the following statements are equivalent :
\begin{enumerate}
\item There exists $s>\frac{1}{2}$ such that $\|u(t)\|_{H^s}$ is unbounded as $t\to \pm \infty$.
\item For \emph{all} $s>\frac{1}{2}$, there exists $C_s,B_s>0$ such that $\|u(t)\|_{H^s}\sim_{t\to \pm\infty} C_se^{B_s|t|}$.
\item The energy and the mass of $u_0$ satisfy the relation
\begin{equation}\label{res}
E=\frac{1}{2}Q^3.
\end{equation}
\end{enumerate}
\end{prop}
This proposition, which will prove the last part of theorem \ref{main}, calls for several comments :
\begin{itemize}
\item On $\mathcal{L}(1)$, there is a dichotomy for solutions of \eqref{quad} : either $u$ remains bounded in every $H^s$, $s\geq 0$, either it blows up in each $H^s$ topology, $s>\frac{1}{2}$, at an exponential rate (but remaining bounded in $H^{1/2}$ and below, of course, since $\mathcal{L}(1)$ is made of $C^\infty$ functions). The question whether such a dichotomy remains true for smooth general solutions, or even on other stable finite-dimensional manifolds (described in proposition \ref{frac}), is left open.
\item Even if the equation we study is only quadratic, hence ``less nonlinear'' in a way than the cubic Szeg\H{o} equation, it appears to be more turbulent, and the example of the turbulent solutions that we exhibit does not display the phenomenon of \emph{backward energy cascade} (as opposed to the behaviour described in \cite[Theorem 1]{livrePG}). In addition, regarding the cubic Szeg\H{o} flow, it is known \cite{explicit} that all trajectories that belong to a finite-dimensional stable manifold are bounded in every $H^s$, $s>\frac{1}{2}$, and the same phenomenon seems to occur in the case of the more physical situation studied in \cite{Bizon}. However, the dichotomy of Proposition \ref{blowup} looks very similar to what Haiyan Xu observed for a \emph{linear perturbation} of the cubic Szeg\H{o} equation. The proof of Proposition \ref{blowup} heavily relies on the technique she developed in \cite{Xu}.
\item Condition \eqref{res} for norm explosion is only expressed in terms of conservation laws of the flow. Therefore, it is compatible with the autonomous nature of equation \eqref{quad}. But more striking is that \eqref{res} is \emph{homogeneous} : if it is satisfied by $u$, then it is also satisfied by $\lambda u$, $\lambda\in\mathbb{C}$. This allows to construct blow-up solutions with arbitrary small initial data in $H^s$, when $s>\frac{1}{2}$ is given, in great contrast with the situation described by H. Xu in \cite{Xu}.

\end{itemize}

This paper is organized as follow : in section \ref{lax-section}, we point out the Lax pair structure of the quadratic Szeg\H{o} equation ; in section \ref{BMO-section}, we prove the existence of the flow of \eqref{quad} on $BMO_+(\mathbb{T})$ ; finally, in section \ref{turbu-section}, we prove the turbulence results contained in Proposition \ref{blowup}.

All of this work has been done under the supervision of Prof. Patrick Gérard, to whom the author is most grateful. He would also like to thank the MSRI in Berkeley, California, which he attended during the Fall semester 2015, and where this study was started.

\vspace*{2em}

\section{The Lax pair structure}\label{lax-section}

\subsection{Smooth solutions}
We begin by proving a useful and standard lemma :

\begin{lemme}\label{lisse}
Let $s>\frac{1}{2}$, and $u_0\in H^s_+(\mathbb{T})$. Then there exists a unique function solution $u\in C^\infty(\mathbb{R},H^s_+(\mathbb{T}))$ such that $u(0)=u_0$ and
\[ i\partial_tu=2J\Pi(|u|^2)+\bar{J}u^2.\]
We also have $\forall t\in\mathbb{R}$, $M(u(t))=M(u_0)$, $Q(u(t))=Q(u_0)$ and $E(u(t))=E(u_0)$.
\end{lemme}

In the sequel, we will refer to these solutions as \emph{smooth} solutions.

\begin{proof}[Proof of the lemma.]
It is well-known that $H^s_+$ is an algebra whenever $s>\frac{1}{2}$. In addition, $|J|\leq \|u\|_{L^3}^3\leq C\|u\|_{H^{1/2}}^3$. By a fixed point argument, we thus get the existence of a time $T>0$ only depending on $\|u_0\|_{H^s}$ for which there is a unique solution of \eqref{quad} starting from $u_0$ in $C([-T,T],H^s_+)$. We observe that $M$, $Q$ and $E$ are conserved along these local-in-time solutions.

The global existence comes from a simple energy estimate. For $u$ solution as above, we have
\begin{equation}\label{majo_brutale}
\frac{d}{dt}\|D^su\|_{L^2}^2\leq C\|u\|_{H^s}(\||u|^2\|_{H^s}+\|u^2\|_{H^s})\leq C\|u\|_{H^s}^2\|u\|_{L^\infty},
\end{equation}
and with the help of the Brezis-Gallouët inequality (see \cite{BrG, ann}) stating that there is $C_s>0$ such that
\[\|u\|_{L^\infty}\leq C_s\|u\|_{H^{1/2}}\sqrt{\log \left( 1+\frac{\|u\|_{H^s}}{\|u\|_{H^{1/2}}}\right)},\]
and also using the boundedness of the $H^{1/2}$ norm of local solutions along time, we get $\frac{d}{dt}\|u\|_{H^s}^2\leq C\|u\|^2_{H^s}\sqrt{\log (1+C'\|u\|_{H^s}^2)}$. Integrating this inequality yields
\[ \|u(t)\|_{H^s}\leq Ce^{Bt^2}.\]
This is enough to prove that $\|u(t)\|_{H^s}$ does not become infinite in finite time, so that local solutions constructed above are in fact global. Global uniqueness then follows from the local argument.
\end{proof}

\subsection{The shifted Hankel operators}

As in \cite{GGtori}, let us now introduce three classes of operators acting on $L^2_+$. For a given $u\in H^{1/2}_+$, we define the \emph{Hankel operator} of symbol $u$ :
\[ H_u:\left\lbrace \begin{aligned}
L^2_+ &\longrightarrow L^2_+, \\
h&\longmapsto \Pi(u\bar{h}),
\end{aligned}\right. \]
and for $b\in L^\infty(\mathbb{T})$, we define as well the \emph{Toeplitz operator} of symbol $b$ by
\[  T_b:\left\lbrace \begin{aligned}
L^2_+ &\longrightarrow L^2_+, \\
h&\longmapsto \Pi(bh),
\end{aligned}\right.\]
Observe that $T_b$ is a $\mathbb{C}$-linear operator, whereas $H_u$ is $\mathbb{C}$-antilinear. The adjoint of $T_b$ is $(T_b)^*=T_{\bar{b}}$, and we have $(H_u(h_1)\vert h_2)=(H_u(h_2)\vert h_1)$, for any $h_1$, $h_2\in L^2_+$. A special Toeplitz operator is the \emph{shift operator}, defined by $S:=T_{e^{ix}}=T_z$.\par
Observe that, for $u\in H^{1/2}_+$, we have the following identity :
\[ H_uS=S^*H_u=H_{S^*u}.\]
This gives rise to the definition of the third class of operators : the \emph{shifted Hankel operator} of symbol $u$ is the operator
\[ K_u:= S^*H_u.\]

The following proposition sums up important properties of $K_u^2$ :

\begin{prop}
For any $u\in H^{1/2}_+$, $K_u^2$ is a $\mathbb{C}$-linear positive self-adjoint operator on $L^2_+$. Moreover, $K_u^2$ is trace class (hence compact).
\end{prop}
The same proposition also holds for $H_u^2$, but will not be needed it here, because $K_u$ turns out to be of more importance in the study of equation \eqref{quad}. Let us just mention that the operator norm of $H_u$ in $\mathcal{L}(L^2_+)$ satisfies $\|H_u\|\leq\|u\|_{H^{1/2}}$. (A proof of this elementary fact can be found \emph{e.g.} in \cite[Appendix A]{thirouin1}.)

\vspace{1em}
We can now state the main theorem of this section :

\begin{thm}[Lax pair for $K_u$]\label{lax_thm}
Let $s>\frac{1}{2}$. Assume $u$ is a smooth solution of \eqref{quad} in $H^s$, as in Lemma \ref{lisse}. Then we have a Lax pair identity :
\begin{equation}\label{lax}
\frac{d}{dt}K_u=B_uK_u-K_uB_u,
\end{equation}
where $B_u:=-i(T_{\bar{J}u}+T_{J\bar{u}})$ is a well-defined anti-self-adjoint operator on $L^2_+$.
\end{thm}

\begin{proof}[Proof]
Let $h\in L_+^2$. We write
\begin{equation*}
i\frac{d}{dt}K_u(h) = \Pi(\bar{z}(i\dot{u})\bar{h})= \Pi(2J\bar{z}|u|^2\bar{h})+\Pi(\bar{J}\bar{z}u^2\bar{h}).
\end{equation*}
Note that the projector $\Pi$ disappeared in the first term, because $\Pi(2J\bar{z}(I-\Pi)(|u|^2)\bar{h})=0$.

Compute in two ways :
\[\Pi(J\bar{z}|u|^2\bar{h})=\Pi(\bar{z}u \overline{\bar{J}uh})=\Pi(\bar{z}u \overline{\Pi(\bar{J}uh)})=K_uT_{\bar{J}u}(h)\]
because $uh=\Pi(uh)$, and
\[ \Pi(J\bar{z}|u|^2\bar{h})=\Pi(J\bar{u}\bar{z}u\bar{h})
=\Pi(J\bar{u}\Pi(\bar{z}u\bar{h}))=T_{J\bar{u}}K_u(h),\]
because $\Pi(J\bar{u}(I-\Pi)(\bar{z}u\bar{h}))=0$.

For the other term, we need an elementary lemma, which can be proved by the simple use of Fourier expansion :
\begin{lemme}\label{barbar}
Given $f\in L^2(\mathbb{T})$, the following identity holds :
\[ (I-\Pi)(\bar{z}f)=\bar{z}\overline{\Pi(\bar{f})}.\]
\end{lemme}

Thus,
\[ \begin{aligned}\Pi(\bar{J}\bar{z}u^2\bar{h})
&= \Pi(\bar{J}u\Pi(\bar{z}u\bar{h}))+\Pi(\bar{J}u(I-\Pi)(\bar{z}u\bar{h}))\\
&= T_{\bar{J}u}K_u(h)+\Pi(\bar{J}\bar{z}u\overline{\Pi(\bar{u}h)})\\
&= T_{\bar{J}u}K_u(h)+K_uT_{J\bar{u}}(h).
\end{aligned}\]

Summing up, and introducing the self-adjoint operator $A_u:=T_{\bar{J}u}+T_{J\bar{u}}$, we have proved that
\[ i\frac{d}{dt}K_u= A_uK_u+K_uA_u.\]
Now, multiply this identity by $-i$, and set $B_u:=-iA_u$. Using the fact that $K_u$ is $\mathbb{C}$-antilinear, we finally get
\begin{equation*}
\frac{d}{dt}K_u=\left[ B_u,K_u \right] = B_uK_u-K_uB_u,
\end{equation*}
which is the claim.
\end{proof}

\begin{cor}\label{equivK}
The eigenvalues $\{\sigma^2_k\}_{k\geq 1}$ (repeated with multiplicity) of $K_u^2$ are conservation laws of equation \eqref{quad}.
\end{cor}

\begin{proof}[Proof]
From \eqref{lax}, we get $\frac{d}{dt}K_u^2=B_uK_u^2-K_u^2B_u$. Now consider the system 
\[ \left\lbrace\begin{aligned}
U'(t)&=B_uU(t), \quad t\in \mathbb{R}\\
U(0)&=I,
\end{aligned}\right. \]
where the unknown $U(t)$ belongs to $\mathcal{L}(L^2_+)$. The existence of a global-in-time solution is ensured by the Cauchy-Lipschitz theorem for \emph{linear} ordinary differential equations. Besides, $U(t)$ is a unitary operator for all time (because of the skew-adjointness of $B_u$). An easy computation then shows that $\frac{d}{dt}(U(t)^*K^2_uU(t))=0$, hence $K_u^2$ remains unitarily equivalent to $K_{u(0)}^2$, and its eigenvalues are conserved.
\end{proof}

\begin{rem}
The evolution of $H_u$ can also be computed :
\[\frac{d}{dt}H_u =B_uH_u-H_uB_u+i\bar{J}(u\vert \cdot)u.\]
This is not a Lax pair : an extra term appears from Lemma \ref{barbar}, because of the zero mode. Nevertheless, if we consider the Hamiltonian $E$ defined on functions on $\mathbb{R}$, instead of $\mathbb{T}$, with a projector $\Pi$ given by
\[ \Pi\left( \int_{-\infty}^{+\infty}\hat{f}(\xi)e^{ix\xi}d\xi \right) = \int_{0}^{+\infty}\hat{f}(\xi)e^{ix\xi}d\xi ,\]
we can also make sense of equation \eqref{quad}, and this time, we would have
\[\frac{d}{dt}H_u =B_uH_u-H_uB_u, \]
with $B_u$ defined as above. This suggests to start a program such as the one developed by Oana Pocovnicu \cite{PocoGrowth, PocoSol} for the cubic Szeg\H{o} equation on the line.
\end{rem}

\subsection{Two consequences : finite dimensional invariant manifolds and \texorpdfstring{$L^\infty$}{Linf} bounds}

One of the main consequences of theorem \ref{lax_thm} is the existence of finite dimensional invariant manifolds for the flow of \eqref{quad}.
\begin{déf}
Let $N$ be a nonnegative integer. We set
\[\mathcal{L}(N):=\left\lbrace u\in H^{1/2}_+(\mathbb{T}) \: \middle| \: \rg K_u =N \right\rbrace. \]
\end{déf}
It happens that we have a complete description of $\mathcal{L}(N)$.

\begin{prop}[see \cite{Xu}]\label{frac}
Fix $N\in\mathbb{N}$. The set $\mathcal{L}(N)$ is exactly the set of all rational functions $u$ which can be written in the form
\[ u(z)=\frac{A(z)}{B(z)}, \quad z\in \mathbb{D},\]
where $A$ and $B$ are complex polynomials of degree at most $N$, satisfying $\deg A=N$ or $\deg B=N$, $A\wedge B=1$, $B(0)=1$, and $B$ having no root in the closed unit disc $\bar{\mathbb{D}}$.
\end{prop}

We see that $\mathcal{L}(N)$ is a manifold of complex dimension $2N+1$, and that each $u\in \mathcal{L}(N)$ belongs to $C^\infty(\mathbb{T})$, hence $u\in H^s_+(\mathbb{T})$ for any $s>\frac{1}{2}$.

\begin{cor}
For each $N\in\mathbb{N}$, the flow of \eqref{quad} preserves the manifold $\mathcal{L}(N)$.
\end{cor}

\begin{proof}[Proof]
Since $\rg K_u = \rg K_u^2$, the rank of $K_u$ is given by the number of non-zero eingenvalues in the list $\{\sigma^2_k\}_{k\geq 1}$. By the preceding corollary, this number is constant for the evolution related to the Hamiltonian $E$.
\end{proof}

Now we turn to another consequence, arguing as in \cite{livrePG} :
\begin{cor}
Let $s>1$ and $u_0\in H^s_+$. Then the solution $u$ of \eqref{quad} starting from $u_0$ satisfies
\[ \sup_{t\in\mathbb{R}}\|u(t)\|_{L^\infty} <+\infty.\]
In addition, for any $s'\in (\frac{1}{2},s]$, there exists $C,B>0$ such that
\[ \forall t\in\mathbb{R},\quad \|u(t)\|_{H^{s'}}\leq Ce^{B|t|}.\]
\end{cor}

\begin{proof}[Proof]
By Peller's theorem \cite[Theorem 1.1 p. 232]{Pe}, we know that
\[\tr |H_u| \sim \|u\|_{B^1_{1,1}}=\sum_{j\in \mathbb{N}}2^j\|\Delta_ju\|_{L^1},\]
where $|H_u|=\sqrt{H_u^2}$, $\tr$ is the trace norm, and $\Delta_ju$ is the $j$-th dyadic piece of $u$. On the other hand, when $s>1$, $H^s\hookrightarrow B^1_{1,1}$, so $\tr |K_u|=\tr |H_{S^*u}|$ remains finite for all time. In fact,
\[ \tr |K_u|=\sum_{k\geq 1} \sqrt{\sigma_k^2}, \]
so by Corollary \ref{equivK}, it is even conserved by the flow. As $B^1_{1,1}\hookrightarrow W$, the Wiener algebra, we find that $\sup_{t\in\mathbb{R}}\|S^*u(t)\|_W\leq C\sup_{t\in\mathbb{R}} \tr |K_u|=C\tr |K_{u_0}|\leq C'\|u_0\|_{H^s} <+\infty$. But naturally $|\hat{u}(0)|\leq \sqrt{Q}$, so that it is also bounded along time. Consequently, $\|u\|_W=|\hat{u}(0)|+\|S^*u\|_W$ remains bounded, and so does $\|u\|_{L^\infty}$ of course.

As for the proof of the a priori estimate, it goes as in Lemma \ref{lisse}, but since $\|u\|_{L^\infty}$ can be replaced by a constant in the r.h.s. of \eqref{majo_brutale}, Gronwall's lemma leads to a simple exponential bound.
\end{proof}

\begin{rem}
Here, we see how crucial the control of the $L^2$ norm is, because it enables to bound the mean of $u$. Taking simply as a Hamiltonian the functional $\tilde{E}(u)=\re (u^2\vert u)=\re J$ (instead of $E=\frac{1}{2}|J|^2$) would break one of the symmetries of the equation, and the associated evolution equation would read $i\partial_t u=2\Pi(|u|^2)+u^2$ (\emph{i.e.} no $J$ factor). This equation also admits a Lax pair for $K_u$, but all non zero solutions blow up in finite time because of the zero mode. Indeed, writing $(u\vert 1)=x+iy$ with $x,y\in\mathbb{R}$, the above equation implies that
\begin{equation*}
\left\lbrace \begin{aligned}
\dot{x}&=2xy, \\
\dot{y}&=y^2-x^2-2Q,
\end{aligned}\right.
\end{equation*}
and the fact that $Q\geq x^2+y^2$ leads to $\dot{y}\leq -y^2-3x^2\leq -y^2$. Besides, we can assume that $y_0:=y(t=0)\neq 0$, since $\dot{y}(t=0)\leq -Q<0$. Thus, $y$ must blow up in finite time $T\in [-\frac{1}{|y_0|},\frac{1}{|y_0|}]$.
\end{rem}

\subsection{The case of \texorpdfstring{$H^1$}{H1} data}
Before going further, we show that an elementary argument can treat the case of $H^1$ initial data, in the spirit of \cite{thirouin1}.

\begin{lemme}\label{h1}
Let $u_0\in H^1_+$, and $u$ be the solution of \eqref{quad} starting from $u_0$. Then for any $s'\in (\frac{1}{2},1]$, there exists constants $C,B>0$ such that
\[ \forall t\in\mathbb{R},\quad \|u(t)\|_{H^{s'}}\leq Ce^{B|t|}.\]
We can choose $C=\|u_0\|_{H^1}$ and $B=B_0(2s'-1)\sqrt{2E}\|u_0\|_{H^{1/2}}$, where $B_0$ is some universal constant.
\end{lemme}

\begin{proof}[Proof]
Since $\|u\|_{H^{1/2}}$ is a bounded quantity, it is enough to show the bound in the case when $s'=1$, by interpolation. Since the $L^2$ norm of $u$ is a constant of motion, we can study the evolution of the $H^1$ norm of $u$ simply by computing :
\[ \begin{aligned}
\frac{d}{dt}\|Du(t)\|_{L^2}^2&=2\re (D\dot{u}|Du)=2\im(2JD(|u|^2)+\bar{J}D(u^2)|Du)\\
&=4\im(JuD\bar{u}+J\bar{u}Du+\bar{J}uDu|Du) \\
&=-4\im (J\Pi(u\overline{Du})|Du) + 4\im \left( J\bar{u}+\bar{J}u\middle| |Du|^2\right).
\end{aligned}\]
Because of the imaginary part, the second term cancels out. Observing that $\Pi(u\overline{Du})=H_u(Du)$, we claim that the modulus of the first term is controlled by
\[ 4|J|\cdot\|\Pi(u\overline{Du})\|_{L^2}\|Du\|_{L^2}\leq 4\sqrt{2E}\|u\|_{H^{1/2}}\|Du\|_{L^2}^2.\]
Indeed $\|H_u\|\leq\|u\|_{H^{1/2}}$ as mentioned above. Since the $H^{1/2}$ norm of $u$ remains uniformly bounded along trajectories, this shows that the $H^1$ norm of $u$ grows at most exponentially in time.
\end{proof}
\vspace*{2em}

\section{The flow on \texorpdfstring{$BMO_+$}{BMO}}\label{BMO-section}
The purpose of this section is to take advantage of the Lax pair structure in order to prove, as in \cite{GKoch}, that the flow of \eqref{quad} can be extended to $BMO_+(\mathbb{T})$. This space can be defined as the intersection of the $BMO$ space of John and Nirenberg with $L^2_+$, or equivalently as the image of $L^\infty(\mathbb{T})$ through $\Pi$ :
\[ BMO_+(\mathbb{T})=\{ \Pi(b)\mid b\in L^\infty(\mathbb{T})\}.\]
The norm is then given by
\[ \|u\|_{BMO}=\inf\{ \|b\|_{L^\infty}\mid \Pi(b)=u\}.\]
Lastly, $BMO_+$ can be seen as the dual of $L^1_+$ (\emph{i.e.} of $L^1$ functions of the torus with vanishing negative frequencies), and thus can also be equipped with the corresponding weak-$*$ topology.

Clearly, if $u=\Pi(b)$ for some $b\in L^\infty$, then for any $h\in L^2_+$, $\Pi(u\bar{h})=\Pi(\Pi(b)\bar{h})=\Pi(b\bar{h})$, so that $\|H_u(h)\|_{L^2}\leq \|b\|_{L^\infty}\|h\|_{L^2}$, hence $H_u$ is continuous on $L^2_+$. Nehari \cite{Nehari} proved that the converse is true : $H_u$ is a continuous operator of $L^2_+$ if and only if $u\in BMO_+$ ; in that case, we have $\|H_u\|=\|u\|_{BMO}$.

The only other fact we will use about $BMO_+$ is that it is continuously embedded in $L^p$ for any $p<\infty$.

\vspace*{1em}
Before turning to the existence of a flow map on $BMO_+$, we would like to prove two crucial lemmas that we will use repeatedly in the sequel.

\begin{lemme}\label{BMO-bound}
Let $u$ be a smooth\footnote{in the sense of Lemma \ref{lisse}, as always.} solution of \eqref{quad}, with $u(0)=u_0$. Then $\forall t\in\mathbb{R}$,
\[ \|u_0\|_{BMO}^2-\|u_0\|_{L^2}^2\leq \|u(t)\|^2_{BMO}\leq \|u_0\|_{BMO}^2+\|u_0\|_{L^2}^2.\]
\end{lemme}

\begin{proof}[Proof.]
From the Lax pair and the proof of Corollary \ref{equivK}, we know that along the evolution, $K_u$ remains unitary equivalent to $K_{u_0}$. In particular, their operator norms are equal. Since $K_u=H_{S^*u}$, we then know that $\|S^*u\|_{BMO}$ is a conserved quantity. Now, for $h\in L^2_+$,
\[ \|H_{S^*u}(h)\|_{L^2}^2=\|S^*H_u(h)\|_{L^2}^2=\|H_u(h)\|_{L^2}^2-|(u\vert h)|^2,\]
and taking the supremum over $h$ with $\|h\|=1$, we get
\[ \|S^*u\|_{BMO}^2\leq \|u\|_{BMO}^2\leq \|S^*u\|_{BMO}^2+\|u\|_{L^2}^2.\]
As $\|u\|_{L^2}^2$ is also conserved, this gives the result.
\end{proof}

Now we can state the main lemma, where the quadratic nature of \eqref{quad} is the most striking :

\begin{lemme}[Lipschitz estimate in $L^2$]\label{lipsch}
Let $u$ and $v$ be two smooth solutions of \eqref{quad}, with $u(0)=u_0$ and $v(0)=v_0$. Then there exists $B$ depending only on $\|u_0\|_{BMO}$ and $\|v_0\|_{BMO}$, such that $\forall t\in\mathbb{R}$,
\begin{equation}\label{l2-lip}
 \|u(t)-v(t)\|_{L^2}\leq e^{B|t|}\|u_0-v_0\|_{L^2}.
\end{equation}
\end{lemme}

\begin{proof}[Proof.]
Write \eqref{quad} as $i\dot{u}=F(u)$, and compute
\[ \frac{d}{dt}\|u-v\|_{L^2}^2=2\im (F(u)-F(v)\vert u-v)=2\int_0^1\im \left( dF(w_\theta)\cdot(u-v)\vert u-v\right) d\theta, \]
where $w_\theta:=\theta u+(1-\theta)v$, for $\theta\in [0,1]$. We have
\[ dF(w)\cdot h= \left( 4(h\vert |w|^2)+2(w^2\vert h) \right) \Pi(|w|^2)+\left( 2(|w|^2\vert h)+(h\vert w^2) \right) w^2 +2T_{\overline{J_w}w+J_w\overline{w}}(h)+2J_wH_w(h).\]
Here, $J_w:=(w^2\vert w)$. Note that the operator $T_{\overline{J_w}w+J_w\overline{w}}$ is self-adjoint. Hence its contribution is cancelled by the imaginary part, and
\[ \frac{d}{dt}\|u-v\|_{L^2}^2=\int_0^1  8\im \left( (w_\theta^2\vert u-v)(|w_\theta|^2\vert u-v) \right) +4\im \left( J_{w_\theta} H_{w_{\theta}}(u-v)\vert u-v \right) d\theta.\]
Straightforward Cauchy-Schwarz inequalities then yield
\[ \left| \frac{d}{dt}\|u-v\|_{L^2}^2\right|\leq \int_0^1  \left(8\|w_\theta\|_{L^4}^4\|u-v\|_{L^2}^2 +4\|w_\theta\|_{L^3}^3 \|w_{\theta}\|_{BMO}\|u-v\|_{L^2}^2 \right) d\theta.\]
Bounding the $L^p$ norms of $w_\theta$ with the $BMO$ norm of $u$ and $v$, and using the preceding lemma, we find a constant $B$ depending on $\|u_0\|_{BMO}$ and $\|v_0\|_{BMO}$ only, and such that
\[ \left| \frac{d}{dt}f(t)\right| \leq Bf(t), \quad f(t):=\|u(t)-v(t)\|_{L^2}^2.\]
Solving the usual Gronwall inequality leads to the statement of the lemma.
\end{proof}

Constructing the flow on $BMO_+$ simply consists in adjusting the strategy of \cite{GKoch}.

\begin{prop}
For every $u_0\in BMO_+$, there exists a unique $u\in C(\mathbb{R},L^2_+)\cap C_{w*}(\mathbb{R},BMO_+)$ solution to
\[ i\partial_tu=2J\Pi(|u|^2)+\overline{J}u^2,\quad u(0)=u_0.\]
This solution stays bounded in the $BMO$ norm : for all $t\in\mathbb{R}$,
\[ \|u_0\|_{BMO}^2-\|u_0\|_{L^2}^2\leq \|u(t)\|^2_{BMO}\leq \|u_0\|_{BMO}^2+\|u_0\|_{L^2}^2.\]
In addition, for each $t\in \mathbb{R}$, the mapping $u_0\mapsto u(t)$ is Lipschitz on the ball $\mathcal{B}_{BMO}(R):=\{v\in BMO_+(\mathbb{T})\mid \|v\|_{BMO}\leq R\}$ endowed with the $L^2$ topology.
\end{prop}

\begin{proof}[Proof.]
Let $u_0\in BMO_+$. We first construct a sequence of smooth functions $u_0^n$ such that
\begin{gather*}
\|u_0^n-u_0\|_{L^2}\longrightarrow 0,\\
\|u_0^n\|_{BMO}\leq \|u_0\|_{BMO}, \quad \forall n\in\mathbb{N}.
\end{gather*}
To do so, let $r<1$, and denote by $P_r(x)=\sum_{n\in\mathbb{Z}}r^{|n|}e^{inx}$ the Poisson kernel. For any $v\in BMO_+$, we have $\|P_r\ast v\|_{L^2}\leq \|v\|_{L^2}$. In addition, $P_r\ast v\in C^\infty_+$, and satisfies $\|P_r\ast v-v\|_{L^2}\to 0$ as $r\to 1^-$. Finally, for $h\in L^2_+$, we have
\[ H_{P_r\ast v}(h)=P_r\ast (H_v(P_r\ast h)),\]
which proves that $\|P_r\ast v\|_{BMO}\leq \|v\|_{BMO}$. Thus we can choose $u_0^n=P_{r_n}\ast u_0$, where $\{r_n\}$ is any sequence of elements of $(0,1)$ converging to $1$.

Now, we denote by $u^n$ the smooth solution of \eqref{quad} such that $u^n(0)=u^n_0$. Fix $T>0$. By Lemma \ref{lipsch}, the sequence $\{u^n\}$ is a Cauchy sequence in $C([-T,T],L^2_+)$, so $\{u^n\}$ converges locally uniformly to a function $u\in C(\mathbb{R},L^2_+)$. In fact, $\{u^n\}$ also converges locally uniformly in any $L^p_+$ (for $2<p<\infty$), since
\[\begin{aligned} \sup_{t\in [-T,T]} \|u^n(t)-u^m(t)\|_{L^p}&\leq \sup_{t\in [-T,T]} \|u^n(t)-u^m(t)\|_{L^2}^{\frac{1}{p}}\|u^n(t)-u^m(t)\|_{L^{2p-2}}^{1-\frac{1}{p}}\\
&\leq C(\|u_0\|_{BMO})e^{CT/p}\|u^n_0-u^m_0\|_{L^2}^{\frac{1}{p}}.
\end{aligned}\]
This allows to pass to the limit in the following expression :
\begin{equation*}
u^n(T)=u^n_0-i\int_0^T \left( 2J_{u^n(s)}\Pi(|u^n(s)|^2)+\overline{J_{u^n(s)}}(u^n(s))^2\right) ds,
\end{equation*}
so that $u$ is a solution of \eqref{quad}. Furthermore, if $t\in\mathbb{R}$ is fixed, then $\forall n\in\mathbb{N}$, $\|u^n(t)\|_{BMO}^2\leq \|u^n_0\|_{BMO}^2+\|u^n_0\|_{L^2}^2\leq \|u_0\|_{BMO}^2+\|u_0\|_{L^2}^2$. Hence, $\{u^n(t)\}$ is a bounded sequence in $BMO_+$, and any of its weak-$*$ cluster point must be $u(t)$, so $u(t)\in BMO_+$ and $u^n(t)\overset{\ast}{\rightharpoonup} u(t)$ as $n\to \infty$. By the principle of uniform boundedness, $\|u(t)\|^2_{BMO}\leq \limsup \|u^n(t)\|^2_{BMO}\leq\|u_0\|_{BMO}^2+\|u_0\|_{L^2}^2$. Finally, the weak-$*$ continuity of $u$ is easy to check, now that we have a uniform bound for the $BMO$ norm of $u$.

The next step consists in proving uniqueness of solutions in $X:=C(\mathbb{R}, L^2_+)\cap C_{w*}(\mathbb{R},BMO_+)$ via an extension of Lemma \ref{lipsch} to non-smooth solutions. This will also imply that the mapping $u_0\mapsto u(t)$ is Lipschitz on balls of $BMO_+$. So let $u_0\in BMO_+$, and let $u$ be the solution of \eqref{quad} constructed above. Suppose that $v$ is another solution starting from $u_0$ in $X$, in the sense that it satisfies $v(0)=u_0$ and $\forall t\in\mathbb{R}$,
\begin{equation}\label{sol-faible}
v(t)=u_0-i\int_0^t \left( 2J_{v(s)}\Pi(|v(s)|^2)+\overline{J_{v(s)}}(v(s))^2\right) ds.
\end{equation}
Fix $T>0$, and let us show that $u=v$ on $[-T,T]$. Observe that $v$ is locally strongly bounded in $BMO$, so the mapping $[-T,T]\to L^2_+$, $t\mapsto 2J_{v(t)}\Pi(|v(t)|^2)+\overline{J_{v(t)}}(v(t))^2$ is (strongly) continuous. So \eqref{sol-faible} can be differentiated, and the same holds for $u$. Thus, we can differentiate $f(t):=\|u(t)-v(t)\|_{L^2}^2$. The rest of the argument of the proof of Lemma \ref{lipsch} remains valid (in particular, the formula for $f'(t)$ is still true), so $f(t)\leq f(0)e^{B|t|}$ on $[-T,T]$, which proves that $u=v$ on $[-T,T]$, therefore on $\mathbb{R}$.

As a consequence, we know that any solution of \eqref{quad} in $X$ is globally bounded in $BMO_+$, so if $u$ and $v$ are two such solutions, we do have a constant $B>0$ only depending on $\|u(0)\|_{BMO}$ and $\|v(0)\|_{BMO}$ such that \eqref{l2-lip} holds for all time.

The last consequence of uniqueness of solutions is the generalization of Lemma \ref{BMO-bound} to solutions of \eqref{quad} in $X$. Indeed, from the above construction, we know that for such solutions the $L^2$ norm is conserved. Besides, for any $t\in\mathbb{R}$, we have shown that $\|u(t)\|^2_{BMO}\leq \|u_0\|_{BMO}^2+\|u_0\|_{L^2}^2$. Now restart a new solution in $X$ whose initial value is $u(t)$, and call it $v$. Uniqueness implies that for all $s\in\mathbb{R}$, $v(s)=u(t+s)$. Hence $\|v(-t)\|^2_{BMO}=\|u_0\|^2_{BMO}\leq \|v(0)\|^2_{BMO}+\|v(0)\|^2_{L^2}=\|u(t)\|^2_{BMO}+\|u_0\|^2_{L^2}$. This finishes to prove that
\[ \forall t\in\mathbb{R}, \quad \left| \|u(t)\|^2_{BMO}-\|u_0\|^2_{BMO}\right| \leq \|u_0\|_{L^2}^2.\]
\end{proof}

The stability properties of the flow of \eqref{quad} on $BMO_+$ can be deduced from the extended version of Lemma \ref{lipsch}.

\begin{prop}[Propagation of low regularity]\label{propa}
Let $0<s<\frac{1}{2}$. Let $u_0\in BMO_+\cap H^s$ and $s'\in (0,s]$. Then the solution $u$ of \eqref{quad} such that $u(0)=u_0$ satisfies
\[ \forall t\in\mathbb{R}, \quad \|u(t)\|_{H^{s'}}\leq \|u_0\|_{H^{s'}}e^{B|t|},\]
where $B$ only depends on $\|u_0\|_{BMO}$ and on $s$. In particular, $u$ stays in $H^s$ for all time.
\end{prop}

\begin{rem}
As recalled in the introduction, such a result is not known in the case of the cubic Szeg\H{o} equation. From \cite[Theorem 3]{GKoch}, it is not clear whether additional regularity propagates or not : at least, it cannot shrink more than exponentially fast.
\end{rem}

\begin{proof}[Proof of Proposition \ref{propa}]
For $u\in BMO_+$ and $y\in \mathbb{R}$, we denote by $\tau_y u$ the function given by $\tau_y u(e^{ix})=u(e^{i(x-y)})$. From the invariances of the equation, if $t\mapsto u(t)$ is a $BMO$ solution to \eqref{quad} starting from $u_0$, then $t\mapsto \tau_y u(t)$ is the solution starting from $\tau_y u_0$. Thus we can apply inequality \eqref{l2-lip} to $u$ and $\tau_y u$ : for all $t\in \mathbb{R}$,
\[ \|(u-\tau_yu)(t)\|_{L^2}\leq e^{B|t|}\|u_0-\tau_yu_0\|_{L^2}.\]
As for any $v\in H^s$,
\begin{equation}\label{Hs-equiv}
\|v\|^2_{H^s}\simeq \|v\|^2_{L^2}+ \int_0^1 \frac{\|v-\tau_yv\|_{L^2}^2}{|y|^{1+2s}}dy,
\end{equation}
it suffices to integrate the preceding inequality in $y$, and we get the result.
\end{proof}

\begin{rem}
Proposition \ref{propa} can be extended to other Besov spaces $B^s_{2,q}$, for $q\in [1,+\infty)$, since we have
\[ \|v\|_{B^s_{2,q}}^q\simeq \|v\|_{L^2}^q+ \int_0^1 \frac{\|v-\tau_yv\|_{L^2}^q}{|y|^{1+qs}}dy.\] 
\end{rem}

\vspace{1em}
A last consequence of \eqref{l2-lip} is that it enables to get the full picture of the maximal growth rate of Sobolev norms of smooth solutions of \eqref{quad}.

\begin{prop}
Let $\frac{1}{2}<s<1$, $u_0\in H^s_+(\mathbb{T})$, and $u$ the solution of \eqref{quad} such that $u(0)=u_0$. Then for any $s'\in (\frac{1}{2},s]$, the $H^{s'}$ norm of $u$ grows at most exponentially in time.
\end{prop}

\begin{proof}[Proof.]
The proof is the same as for proposition \ref{propa}, since \eqref{Hs-equiv} holds for any $s\in (0,1)$.
\end{proof}

\vspace*{2em}
\section{Turbulent trajectories : an explicit computation in \texorpdfstring{$\mathcal{L}(1)$}{L1}}\label{turbu-section}

In this section, we follow the intuition of \cite{Xu} : we study the flow of \eqref{quad} on the manifold $\mathcal{L}(1)$, and we find the necessary and sufficient condition for weak turbulence to occur as stated in Proposition \ref{blowup}.

In view of Proposition \ref{frac}, any element $u\in\mathcal{L}(1)$ can be written as
\begin{equation}\label{L1} u(z)= b+\frac{cz}{1-pz}, \quad z\in\mathbb{D},
\end{equation}
for some $b$, $c$, $p\in\mathbb{C}$ satisfying also $|p|<1$ and $c\neq 0$. Let us rephrase the evolution of \eqref{quad} in these coordinates $(b,c,p)$.

If $u$ is of the form \eqref{L1},
\[ J=\int_\mathbb{T}|u|^2u =\frac{1}{2\pi}\int_0^{2\pi} \left( b+\frac{ce^{ix}}{1-pe^{ix}} \right) ^2\left( \bar{b}+\frac{\bar{c}}{e^{ix}-\bar{p}}\right) dx,\]
and by the residue theorem,
\begin{equation}\label{L1_J}
J=|b|^2b+\frac{2b|c|^2}{1-|p|^2}+\frac{|c|^2c\bar{p}}{(1-|p|^2)^2}.
\end{equation}
Similarly, we compute the three main conservation laws :
\begin{gather*}
Q=|b|^2+\frac{|c|^2}{1-|p|^2}, \\
M=\frac{|c|^2}{(1-|p|^2)^2}, \\
E=\frac{1}{2}|b|^6+2|b|^4\frac{|c|^2}{1-|p|^2}+\frac{|b|^2|c|^2}{(1-|p|^2)^2}\left( \re  (b\bar{c}p) +2|c|^2 \right)+2\frac{|c|^4}{(1-|p|^2)^3}\re (b\bar{c}p) +\frac{1}{2}\frac{|p|^2|c|^6}{(1-|p|^2)^4}.
\end{gather*}

By a partial fraction decomposition, we also know that
\[ \Pi (|u|^2)=\Pi \left( \left( b+\frac{cz}{1-pz} \right) \left( \bar{b}+\frac{\bar{c}}{z-\bar{p}}\right)\right) = |b|^2+\frac{|c|^2}{1-|p|^2}+\left( \frac{p|c|^2}{1-|p|^2}+\bar{b}c\right) \frac{z}{1-pz},\]
so that equation \eqref{quad} on $\mathcal{L}(1)$ finally reads
\[\left\lbrace
\begin{aligned}
i\dot{p} &=c\bar{J},\\
i\dot{c} &=2bc\bar{J}+2\bar{b}cJ + \frac{2Jp|c|^2}{1-|p|^2},\\
i\dot{b} &=b^2\bar{J}+2|b|^2J+\frac{2J|c|^2}{1-|p|^2},
\end{aligned}\right.
\]
with $J$ given by \eqref{L1_J}.

\vspace{1em}
Now we can turn to the
\begin{proof}[Proof of Proposition \ref{blowup}]
We begin by proving that \textit{(i)} implies \textit{(iii)}. Suppose that there exists $\underline{s}>\frac{1}{2}$, as well as a sequence of times $\{t_n\}_{n\in\mathbb{N}}$, such that $\|u(t_n)\|_{H^{\underline{s}}} \to \infty$ as $n\to \infty$. In terms of coordinates $(b,c,p)$, since $|b|$ is bounded (by $Q$), this means that
\[ \sum_{k=1}^\infty |c(t_n)|^2|p(t_n)|^{2k} (1+k^2)^{\underline{s}} \underset{n\to+\infty}{\longrightarrow} \infty.\]
Since $|c|$ is a bounded function (by $\sqrt{M}$, for instance), we see that necessarily, $|p(t_n)|\to 1$ when $n\to \infty$. Now, in view of the expression of the momentum, $|c(t_n)|=\sqrt{M}\cdot (1-|p(t_n)|^2)$ goes to $0$. As a consequence, writing
\[ Q=|b(t_n)|^2+\sqrt{M}|c(t_n)|,\]
we see that $|b(t_n)|^2\to Q$ as $n\to\infty$. If we move on to the energy, we get
\[ E=\frac{1}{2}|J(t_n)|^2= \frac{1}{2}\left| |b(t_n)|^2b(t_n)+2\sqrt{M}b(t_n)|c(t_n)|+Mc(t_n)\overline{p(t_n)} \right|^2,\]
so that $\frac{1}{2}|b(t_n)|^6\to E$ as $n\to\infty$. Hence $E=\frac{1}{2}Q^3$.

To conclude the proof, it suffices to show that \textit{(iii)} implies \textit{(ii)}, so we assume \eqref{res}. Developing the expression of $Q$, we get in $(b,c,p)$ coordinates :
\[ |b|^4+\frac{|b|^2|c|^2}{1-|p|^2}+2\re \left( \frac{b\bar{c}p}{1-|p|^2}\right) \left(|b|^2+\frac{2|c|^2}{1-|p|^2}\right) = \frac{|c|^4(1-2|p|^2)}{(1-|p|^2)^3},\]
or equivalently, using the conservation laws :
\begin{equation}\label{re_bcp}
Q(Q-|c|\sqrt{M})+ 2(Q+|c|\sqrt{M}) \re \left( \frac{b\bar{c}p}{1-|p|^2}\right) =2|c|^2M-|c|M\sqrt{M}.
\end{equation}
We would like to show that $|p(t)|\to 1$, or that $|c(t)|\to 0$ as $t\to +\infty$ (the negative times are treated in the same way). First of all, notice that $c$ never cancels out, because $u$ must stay in $\mathcal{L}(1)$. Consequently, $t\mapsto |c(t)|$ is a smooth function of time.

Let us compute, from the equation on $c$,
\[\frac{d|c|}{dt}= \frac{1}{|c|}\re (\bar{c}\dot{c})=\frac{2|c|}{1-|p|^2}\left[ |b|^2+\frac{2|c|^2}{1-|p|^2} \right]\im (b\bar{c}p)=2|c|(Q+|c|\sqrt{M})\im \left( \frac{b\bar{c}p}{1-|p|^2}\right) ,\]
so that, by \eqref{re_bcp},
\begin{align*}
\left( \frac{1}{|c|}\frac{d|c|}{dt}\right)^2&=4(Q+|c|\sqrt{M})^2\left[ \im \left(\frac{b\bar{c}p}{1-|p|^2}\right)\right]^2 \\
&=4(Q+|c|\sqrt{M})^2\frac{|b|^2|c|^2|p|^2}{(1-|p|^2)^2} - \left( 2|c|^2M-|c|M\sqrt{M} -Q(Q-|c|\sqrt{M}) \right)^2 \\
&=4(Q+|c|\sqrt{M})^2(Q-|c|\sqrt{M})(M-|c|\sqrt{M}) - (2M|c|^2+\sqrt{M}(Q-M)|c|-Q^2)^2
\end{align*}
Expanding this polynomial in $|c|$, one realizes that the terms in $|c|^4$ and $|c|^3$ cancel out. In the end,
\begin{equation}\label{haiyan} \left( \frac{1}{|c|}\frac{d|c|}{dt}\right)^2 = -M(M+Q)^2|c|^2+2Q^2\sqrt{M}(M-Q)|c|+Q^3(4M-Q)=\mathcal{P}(|c|\sqrt{M}),
\end{equation}
where $\mathcal{P}(X):=-(M+Q)^2X^2+2Q^2(M-Q)X+Q^3(4M-Q)$. The discriminant of $\mathcal{P}$ is
\[\Delta = Q^4(M-Q)^2+(M+Q)^2Q^3(4M-Q)= 8M^2Q^4+4M^3Q^3 >0,\]
so $\mathcal{P}$ has two distinct roots :
\[r_\pm := \frac{Q^2(M-Q)\pm 2MQ\sqrt{2Q^2+MQ}}{(M+Q)^2}.\]

Observe that $\mathcal{P}$ takes nonnegative values between $r_-$ and $r_+$ only. In view of the differential equation \eqref{haiyan}, and since $|c|>0$, we know that $\mathcal{P}(x)$ must be nonnegative at least for one $x>0$. Therefore we must have $r_+> 0$, and thus
\[ \begin{aligned}2M& \sqrt{2Q+M}> \sqrt{Q}(Q-M) \\
&\Leftrightarrow (Q\leq M) \text{ or } (8M^2Q+4M^3> Q^3-2Q^2M+QM^2) \\
&\Leftrightarrow (Q\leq M) \text{ or } ((4M-Q)(Q+M)^2 > 0)
\end{aligned}\]
So this proves that as a consequence of \eqref{res}, we must have $Q<4M$.

It follows that $r_-<0$. Indeed, the product $r_+r_-$ is given by to coefficients of $\mathcal{P}$, namely
\[ r_+r_-= \frac{Q^3(4M-Q)}{-(M+Q)^2}<0.\]
Therefore we can factorize $\mathcal{P}(X)=(M+Q)^2(X-r_-)(r_+-X)$, and we find
\[ \left( \frac{d |c|}{dt}\right)^2= (M+Q)^2|c|^2(|c|\sqrt{M}-r_-)(r_+-|c|\sqrt{M}).\]

Suppose that for some time $t_0\in \mathbb{R}$,
\[ \frac{d |c|}{dt}(t_0)= +(M+Q)|c(t_0)|\sqrt{(|c(t_0)|\sqrt{M}+|r_-|)(r_+-|c(t_0)|\sqrt{M})},\]
then $|c|$ increases and the equality holds true until $|c|$ reaches the value $\frac{r_+}{\sqrt{M}}$. This must happen in finite time, since for $t\geq t_0$,
\[ \frac{d|c|}{dt}(t)\geq (M+Q)|c(t_0)|\sqrt{|r_-|}\cdot \sqrt{r_+-|c(t)|\sqrt{M}},\]
so that there exists $K>0$ such that $\frac{d}{dt}\sqrt{r_+-|c(t)|\sqrt{M}}\leq -K$. Without loss of generality, we assume that $|c(t)|=\frac{r_+}{\sqrt{M}}$ at $t=0$. Furthermore,
\[ \frac{d^2|c|}{dt^2}(0)=-\frac{1}{2}\sqrt{M}(M+Q)\frac{r_+^2}{M}(r_++|r_-|)<0,\]
thus $|c|$ must decrease immediately after time $0$. This proves that for all $t\geq 0$,
\[ \frac{d |c|}{dt}= -(M+Q)|c|\sqrt{(|c|\sqrt{M}+|r_-|)(r_+-|c|\sqrt{M})}.\]

It appears that this equation can be integrated. To simplify notations, set $f:=|c|\sqrt{M}$. We then have
\[ \int_{f(0)}^{f(t)} \frac{df}{f\sqrt{f+|r_-|}\sqrt{r_+-f}}=-(M+Q)t. \]
Making the change of variables $y=\sqrt{\frac{r_+-f}{f+|r_-|}}$ in the integral yields, for some constant $C\in\mathbb{R}$,
\[ \frac{1}{\sqrt{|r_-|r_+}}\log \left( \frac{2}{1+\sqrt{\frac{|r_-|}{r_+}}\sqrt{\frac{r_+-f(t)}{f(t)+|r_-|}}} -1\right)=C-(M+Q)t. \]
In view of the coefficients of $\mathcal{P}$, we know that $|r_-|r_+=Q^3(4M-Q)(Q+M)^{-2}$, so we set $\kappa:=Q^{3/2}\sqrt{4M-Q}$, and we finally have
\[ |c(t)|=\frac{f(t)}{\sqrt{M}}=\frac{C'}{|r_-|(1+Ce^{-\kappa t})^2+r_+(1-Ce^{-\kappa t})^2}e^{-\kappa t}.\]
In the end, this proves that $|c(t)|\sim_{t \to +\infty} Ce^{-\kappa t}$. In view of our preliminary remark, we have in fact $|c(t)|\sim_{t \to \pm\infty} Ce^{-\kappa |t|}$. As a consequence, $|p(t)|$ goes to $1$ in both time directions.

It remains to compute the $H^s$ norm of $u$ for $s>\frac{1}{2}$. Expanding $u$ in Fourier series, and noticing again that $b$ is bounded, we only need to estimate the sum
\[ |c(t)|^2 \sum_{k=1}^\infty |p(t)|^{2k} \cdot k^{2s}\]
as $t\to\pm\infty$. It is a classical result that the power series $\sum_{k=1}^\infty x^kk^\alpha$ is equivalent to $C(1-x)^{-(1+\alpha)}$ when $x\to 1^-$. So
\[ \|u(t)\|_{H^s}^2\sim C\frac{|c|^2}{(1-|p|^2)^{1+2s}}=CM^{\frac{1}{2}+s}|c|^{1-2s},\]
\textit{i.e.} $\|u(t)\|_{H^s}^2\sim C_se^{(2s-1)\kappa |t|}$, which was the claim.
\end{proof}

\begin{rem}
The coefficient $\kappa$ we find in the proof of theorem \ref{blowup} is very similar to the one of the a priori estimate of Lemma \ref{h1}. Indeed, because of assumption \eqref{res},
\[ \kappa = Q^{3/2}\sqrt{4M-Q}=\sqrt{2E}\cdot \sqrt{4M-Q}, \]
and the factor $\sqrt{4M-Q}$ is ``not far'' from $\|u_0\|_{H^{1/2}}$.
\end{rem}


\bibliography{mabiblio}

\begin{thebibliography}{10}

\bibitem{Bizon}
P.~Bizo\'n, B.~Craps, O.~Evnin, D.~Hunik, V.~Luyten, and M.~Maliborski.
\newblock Conformal flow on {$\rm S^3$} and weak field integrability in {$\rm
  AdS_4$}.
\newblock {\em Comm. Math. Phys.}, 353(3):1179--1199, 2017.

\bibitem{Bourgain}
J.~Bourgain.
\newblock On the growth in time of higher {S}obolev norms of smooth solutions
  of {H}amiltonian {PDE}.
\newblock {\em Internat. Math. Res. Notices}, (6):277--304, 1996.

\bibitem{BrG}
H.~Brezis and T.~Gallou\"et.
\newblock Nonlinear {S}chr\"odinger evolution equations.
\newblock {\em Nonlinear Anal.}, 4(4):677--681, 1980.

\bibitem{ann}
P.~G{\'e}rard and S.~Grellier.
\newblock The cubic {S}zeg{\H o} equation.
\newblock {\em Ann. Sci. \'Ec. Norm. Sup\'er. (4)}, 43(5):761--810, 2010.

\bibitem{GGtori}
P.~G{\'e}rard and S.~Grellier.
\newblock Invariant tori for the cubic {S}zeg{\H o} equation.
\newblock {\em Invent. Math.}, 187(3):707--754, 2012.

\bibitem{explicit}
P.~G{\'e}rard and S.~Grellier.
\newblock An explicit formula for the cubic {S}zeg{\H{o}} equation.
\newblock {\em Transactions of the American Mathematical Society},
  367(4):2979--2995, 2015.

\bibitem{livrePG}
P.~G\'erard and S.~Grellier.
\newblock The cubic {S}zeg{\H o} equation and {H}ankel operators.
\newblock {\em Ast\'erisque}, (389):vi+112, 2017.

\bibitem{GKoch}
P.~G\'erard and H.~Koch.
\newblock The cubic {S}zeg{\H o} flow at low regularity.
\newblock In {\em S\'eminaire {L}aurent {S}chwartz---\'Equations aux
  d\'eriv\'ees partielles et applications. {A}nn\'ee 2016--2017}, pages Exp No.
  XIV, 14 p. Ed. \'Ec. Polytech., Palaiseau, 2017.

\bibitem{hani}
Z.~Hani.
\newblock Long-time instability and unbounded {S}obolev orbits for some
  periodic nonlinear {S}chr\"odinger equations.
\newblock {\em Arch. Ration. Mech. Anal.}, 211(3):929--964, 2014.

\bibitem{HPTV}
Z.~Hani, B.~Pausader, N.~Tzvetkov, and N.~Visciglia.
\newblock Modified scattering for the cubic {S}chr\"odinger equation on product
  spaces and applications.
\newblock {\em Forum Math. Pi}, 3:e4, 63, 2015.

\bibitem{Nehari}
Z.~Nehari.
\newblock On bounded bilinear forms.
\newblock {\em Annals of Mathematics}, 65:153--162, 1957.

\bibitem{Pe}
V.~Peller.
\newblock {\em Hankel operators and their applications}.
\newblock Springer Science \& Business Media, 2012.

\bibitem{planchon}
F.~Planchon, N.~Tzvetkov, and N.~Visciglia.
\newblock On the growth of {S}obolev norms for {NLS} on 2- and 3-dimensional
  manifolds.
\newblock {\em Anal. PDE}, 10(5):1123--1147, 2017.

\bibitem{PocoGrowth}
O.~Pocovnicu.
\newblock Explicit formula for the solution of the {S}zeg{\H o} equation on the
  real line and applications.
\newblock {\em Discrete Contin. Dyn. Syst.}, 31(3):607--649, 2011.

\bibitem{PocoSol}
O.~Pocovnicu.
\newblock Traveling waves for the cubic {S}zeg{\H o} equation on the real line.
\newblock {\em Anal. PDE}, 4(3):379--404, 2011.

\bibitem{Staffilani}
G.~Staffilani.
\newblock On the growth of high {S}obolev norms of solutions for {K}d{V} and
  {S}chr\"odinger equations.
\newblock {\em Duke Math. J.}, 86(1):109--142, 1997.

\bibitem{thirouin1}
J.~Thirouin.
\newblock On the growth of {S}obolev norms of solutions of the fractional
  defocusing {NLS} equation on the circle.
\newblock {\em Annales de l'Institut Henri Poincare (C) Non Linear Analysis},
  34(2):509 -- 531, 2017.

\bibitem{Xu}
H.~Xu.
\newblock Large time blowup for a perturbation of the cubic {S}zeg{\H{o}}
  equation.
\newblock {\em Analysis \& PDE}, 7(3):717--731, 2014.

\bibitem{Xu2}
H.~Xu.
\newblock Unbounded {S}obolev trajectories and modified scattering theory for a
  wave guide nonlinear {S}chr{\"o}dinger equation.
\newblock {\em Mathematische Zeitschrift}, 286(1):443--489, June 2017.

\end{thebibliography}
\bibliographystyle{plain}

\vspace{0,5cm}
\textsc{Département de mathématiques et applications, École normale supérieure,
CNRS, PSL Research University, 75005 Paris, France}

\textit{E-mail address: }\texttt{joseph.thirouin@ens.fr}

\end{document}